\documentclass{article}
\usepackage[margin=1in]{geometry}

\AtEndDocument{\bigskip{\footnotesize%
  \textsc{Department of Applied Mathematics, Delhi Technological University, Delhi–110042, India} \par
  \textit{E-mail address:} \texttt{spkumar@dtu.ac.in} \par
  \addvspace{\medskipamount}
  \textsc{Department of Applied Mathematics, Delhi Technological University, Delhi–110042, India} \par  
  \textit{E-mail address:} \texttt{suryagiri456@gmail.com} \par
}}
\usepackage{graphicx}
\usepackage{hyperref}
\usepackage[caption = false]{subfig}
\usepackage{amsthm}
\usepackage{amssymb}
\usepackage{caption}
\usepackage{mathrsfs}
\usepackage{mathtools}
\usepackage{enumerate}
\usepackage{amssymb}
\usepackage{wrapfig}
\usepackage{float}
\allowdisplaybreaks

\usepackage{geometry}
\numberwithin{equation}{section}
\DeclareMathOperator{\RE}{Re}

\theoremstyle{plain}
\usepackage{lipsum}

\newtheorem{theorem}{Theorem}[section]
\newtheorem{corollary}[theorem]{Corollary}
\newtheorem{example}[theorem]{Example}
\newtheorem{assumption}[theorem]{Assumption}
\newtheorem{lemma}{Lemma}[section]

\theoremstyle{definition}

\theoremstyle{remark}
\newtheorem{remark}{Remark}[section]

\makeatother

\usepackage{authblk}
\begin{document}
\title{Sharp Bounds of Fifth Coefficient and Hermitian-Toeplitz determinants for Sakaguchi Classes}
\author{Surya Giri and S. Sivaprasad Kumar }


\date{}


	

\maketitle	
	
\begin{abstract} 
	For the classes of analytic functions $f$ defined on the unit disk satisfying
      $$\frac{z {f}'(z)}{f(z) - f(-z)}  \prec \varphi(z) \quad \text{and} \quad \frac{(2 z {f}'(z))'}{(f(z) - f(-z))'}  \prec \varphi(z),   $$
      denoted by $\mathcal{S}^*_s(\varphi)$ and $\mathcal{C}_s(\varphi)$ respectively, the sharp bound of the $n^{th}$ Taylor coefficients are known for $n=2,$ $3$ and $4$. In this paper, we obtain the sharp bound of the fifth coefficient. Additionally, the sharp lower and upper estimates of the third order Hermitian Toeplitz determinant for the functions belonging to these classes are determined. The applications of our results lead to the establishment of certain
      new and previously known results.
\end{abstract}
\vspace{0.5cm}
\maketitle

\section{Introduction}
    Let $\mathcal{H}$ be the class of holomorphic functions in the unit disk $\mathbb{D}$ and $A\subset \mathcal{H}$ represent the class of functions $f$ satisfying $f(0)= f'(0)-1=0$. Let $\mathcal{S}\subset \mathcal{A}$ be the class of univalent functions.
    A function $f\in \mathcal{H}$ is said to be starlike with respect to symmetric point if for $r$ less than and sufficiently close to 1 and every $z_0$ on $\lvert z \rvert =r$, the angular velocity of $f(z)$ about the point $f(-z_0)$ is positive at $z = z_0$ as $z$ traverses the circle $\lvert z \rvert =r$ in the positive direction. Sakaguchi \cite{Saka}  showed that a function $f \in \mathcal{A}$ is starlike with respect to symmetrical point if and only if
    $$ \RE \frac{z {f}'(z)}{f(z) - f(-z)} > 0. $$
    The class of all such functions is denoted by $\mathcal{S}^*_s$. It is noted that the class of functions univalent and starlike with respect to symmetric points includes the classes of convex functions and odd functions starlike with respect to the origin \cite{Saka}.
    Afterwards, Das and Singh \cite{DasSingh} introduced the class $\mathcal{K}_s$ of $f\in \mathcal{A}$, known as convex functions with respect to symmetric points, which satisfy
    $$ \RE \frac{(2 z f'(z))'}{(f(z) - f(-z))'} >0. $$
    The functions in the class are convex and Das and Singh proved that the $n^{th}$ coefficient of functions in $\mathcal{K}_s$ is bounded by $1/n, n\geq 2.$

    Incorporating the notion of subordination, Ravichandran \cite{ravibhai} generalized these classes as
    $$ \mathcal{S}^*_s(\varphi) = \bigg\{ f \in \mathcal{A}: \frac{z {f}'(z)}{f(z) - f(-z)} \prec \varphi(z) \bigg\}, $$
        $$ \mathcal{C}_s(\varphi) = \bigg\{ f \in \mathcal{A}: \frac{(2 z f'(z))'}{(f(z) - f(-z))'} \prec \varphi(z) \bigg\}, $$
    where $\varphi(z)$ is analytic univalent function in $\mathbb{D}$, whose domain is symmetric with respect to $\varphi(0)=1$ and $\RE\varphi(z)>0$. Let us take\
\begin{equation}\label{phi(z)}
      \varphi(z) = 1+ B_1 z + B_2 z^2 + B_3 z^3 +\cdots , \quad B_1 >0.
\end{equation}
    They obtained certain convolution conditions and growth and distortion estimates for functions belonging to these classes. Later, Shanmugam et al. \cite{Shan} found the sharp bound of Feketo-Szeg\"{o} functional, $\lvert a_3 - \mu a_2 ^2 \rvert$ for the classes $\mathcal{S}^*_s(\varphi)$ and $\mathcal{C}_s(\varphi)$, which easily provides the bound for initial coefficients $\lvert a_2 \rvert$ and $\lvert a_3\rvert$. Further, the sharp bound of $\lvert a_4\rvert$ was determined by Khatter et al. \cite{Khatt} and for certain important choices of $\varphi$ such as
\begin{equation}\label{SSs}
\left.
\begin{array}{cc}
\begin{aligned}
      \mathcal{S}^*_{s,e} &:=  \mathcal{S}^*(e^z),  \quad \mathcal{S}^*_{s,L}:= \sqrt{1+z} \quad \text{and} \\
      \mathcal{S}^*_{s,RL} &:= \mathcal{S}^*_s \big(\sqrt{2}-(\sqrt{2}-1)\sqrt{(1-z)/(1+2 (\sqrt{2}-1)z)}\big),
\end{aligned}
\end{array}
\right\}
\end{equation}
    the sharp bound of $\lvert a_5 \rvert$ was also established. The sharp bound of $\lvert a_5 \rvert$ for functions belonging to the classes $\mathcal{S}^*_{s}(\varphi)$ and $\mathcal{C}_{s}(\varphi)$ was still unknown. We get this bound in section \ref{fifthSS}. Recently, Gangania and Kumar \cite{GKSs} studied a generalized Bohr Rogosinski type inequalities for the classes $\mathcal{S}^*_{s}(\varphi)$ and $\mathcal{C}_{s}(\varphi)$. Kumar and Kumar \cite{KK} obtained the sharp bound of second and third order Hermitian-Toeplitz determinant for Sakaguchi functions and the classes defined in (\ref{SSs}).

    For $f\in \mathcal{A}$ and $m,n \in \mathbb{N}$, the Hermitian-Toeplitz determinant of order $n$ is given by
\begin{equation}\label{tplz}
      T_{m}(n)(f)= \begin{vmatrix}
	a_n & a_{n+1} & \cdots & a_{n+m-1} \\
	{a}_{n+1} & a_n & \cdots & a_{n+m-2}\\
	\vdots & \vdots & \vdots & \vdots\\
    {a}_{n+m-1} & {a}_{n+m-2} & \cdots & a_n
	\end{vmatrix},
\end{equation}
    It can be easily seen that the determinant of $T_{m,1}(f)$ is rotationally invariant that is determinant of $T_{m,1}(f)$ and $T_{m,1}(f_\theta)$ are same, where $f_\theta=e^{-i \theta}f(e^{i\theta }z)$ and $\theta\in \mathbb{R}$. Since for $n=1$ and $f\in \mathcal{A}$, $a_1=1$. Thus, the third order Hermitian-Toepilitz determinant is
\begin{equation}\label{T31}
     T_{3,1}(f)= 1 - 2\lvert a_2\rvert^2 + 2 \RE \left(a_2^{2} \bar{a}_3\right)- \lvert a_3\rvert^2.
\end{equation}
    Ye and Lim \cite{LHLIM} proved that any $n \times n$ matrix over $\mathbb{C}$ generically can be written as the product of some Toeplitz matrices or Hankel matrices. The applications of Toeplitz matrices and Toeplitz determinants can be seen in the field of pure as well as applied mathematics. They arise in algebraic geometry,  numerical integration, numerical integral equations and queueing networks. For more applications, we refer \cite{LHLIM} and the references cited therein.

    Numerous papers have recently focused on finding the sharp upper and lower bounds  of the Hermitian Toeplitz determinants for functions in $\mathcal{A}$.
     Cudna et al. \cite{cudna} initiated this  work by determining the sharp lower and upper estimate for $ T_{2,1}(f)$ and $T_{3,1}(f)$ for the class of starlike and convex functions of order $\alpha$, $0\leq \alpha<1$. The bounds of  $ T_{2,1}(f)$ and $T_{3,1}(f)$ for the class $\mathcal{S}$ and its certain subclasses were derived by Obradovi\'{c} and Tuneski \cite{tuneski}. For more recent work on this topic, we refer \cite{ahuja,Lecko,Lecko2,Vkumar2} and the references cited therein.

    The aim of this paper is to derive the bound of $\lvert a_5 \rvert$  and  third order Hermitian Toeplitz determinant for $f$ belonging to the classes $ \mathcal{S}^*_{s}(\varphi)$ and $\mathcal{C}_{s}(\varphi)$.
\section{Fifth Coefficient Bound}\label{fifthSS}
    Let $\mathcal{P}$ be the class of Carath\'{e}odory functions $p(z) = 1+ \sum_{n=1}^\infty p_n z^n$ satisfying $\RE p(z)>0$ $(z\in \mathbb{D})$. The subsequent lemmas are used in order to prove the bound of $\lvert a_5\rvert$.
\begin{lemma}\label{L2}\cite{23}
     If the functions
     $ 1 + \sum_{n=1}^\infty p_n z^n$ and  $1 + \sum_{n=1}^\infty q_n z^n $
     are members of $\mathcal{P}$, then the same is true of the function
     $$ 1 + \sum_{n=1}^\infty \frac{p_n q_n}{2} z^n.$$
\end{lemma}
\begin{lemma}\label{A_4}\cite{23}
   Let $h(z)=1 + \beta_1 z + \beta_2 z^2 + \cdots $ and $1+H(z)=1 + b_1 z + b_2 z^2 + \cdots$ be functions in $\mathcal{P}$, and set
     $$\gamma_n = \frac{1}{2^n} \left[ 1 + \frac{1}{2}\sum_{\nu=1}^n {n \choose \nu} \beta_\nu \right], \quad \gamma_0=1.$$
     If $A_n$ is defined by
     $$\sum_{n=1}^\infty (-1)^{n+1}\gamma_{n-1}H^n (z)=\sum_{n=1}^\infty A_n z^n,$$
     then $\lvert A_n\rvert\leq 2.$
\end{lemma}
\begin{lemma}\label{Cho's lemma}
    \cite{24} If $p(z)=  1 + \sum_{n=1}^\infty p_n z^n \in\mathcal{P}$, then for some $\xi_1,\xi_2,\xi_3\in \overline{\mathbb{D}}$,
\begin{equation}\label{tgther}
\left.
\begin{array}{cc}
\begin{aligned}
     p_1 &=2\xi_1,  \quad  p_2=2\xi_1^2+2(1- \lvert \xi_1 \rvert^2)\xi_2, \\
    p_3 &=2\xi_1^3+4(1- \lvert \xi_1 \rvert^2)\xi_1 \xi_2-2(1- \lvert \xi_1 \rvert^2)\overline{\xi_1}\xi_2^2+2(1- \lvert \xi_1 \rvert^2)(1- \lvert \xi_2 \rvert^2)\xi_3.
\end{aligned}
\end{array}
\right\}
\end{equation}
    Further, for $\xi_1,\xi_2\in\mathbb{D}$ and $\xi_3\in\mathbb{T}$, there is a unique function $p(z) = (1 + \omega(z))/(1-\omega(z)) \in\mathcal{P}$ with $p_1,$ $p_2$ and $p_3$ as
    in (\ref{tgther}), where
\begin{equation}\label{omega_2(z)}
    \omega(z)=z\Psi_{-\xi_1}(z\Psi_{-\xi_2}(\xi_3 z)),
\end{equation}
    that is
    $$p(z)=\frac{1+(\overline{\xi_2}\xi_3+\overline{\xi_1}\xi_2+\xi_1)z+(\overline{\xi_1}\xi_3+\xi_1\overline{\xi_2}\xi_3+\xi_2)z^2+\xi_3 z^3}{1+(\overline{\xi_2}\xi_3+\overline{\xi_1}\xi_2-\xi_1)z+(\overline{\xi_1}\xi_3-\xi_1\overline{\xi_2}\xi_3-\xi_2)z^2-\xi_3 z^3}.$$
\end{lemma}
      Conversely, for given $\xi_1, \xi_2 \in \mathbb{D}$ and $\xi_3\in\overline{\mathbb{D}}$, we can construct a (unique) function $p(z)= 1 + \sum_{n=1}^\infty p_n z^n \in\mathcal{P}$ such that $p_1$, $p_2$ and $p_3$ satisfy the identities in (\ref{tgther}). For this, we define
\begin{equation}\label{omega_3(z)}
    \omega(z)=\omega_{\xi_1, \xi_2, \xi_3} (z)= z \Psi_{-\xi_1}(z \Psi_{-\xi_2}(\xi_3 z)).
\end{equation}
      Moreover, if we define $p(z)=(1+\omega(z))/(1-\omega(z))$, then $p_1$, $p_2$ and $p_3$ satisfy the identities in (\ref{tgther}) (see the proof of \cite[Lemma 2.4]{24}).

\begin{assumption}\label{prp1}  Let $\varphi(z)$ be given by (\ref{phi(z)}). The following conditions on coefficients of $\varphi$ helps us to prove the result.\\
\begin{align*}
     {\bf{C1}}: &  \; \lvert B_1^3 - 2 B_1 B_2 + 2 B_2^2 \rvert < \lvert 2 B_1^2 - B_1^3 - 2 B_1 B_2  \rvert, \\
     {\bf{C2}}: & \; \lvert B_1^3 - B_1^2 B_2 + 3 B_2^2 - 3 B_1 B_3 \rvert < 3 \lvert  B_1^3 -B_1^2  + B_2^2 \rvert, \\
     {\bf C3}:
     &\; \lvert B_1^7 - B_1^6 ( 8 B_2 + 3) -  6 B_1^4 ( B_2 ( 3 B_2 + 2 B_3  + 2 ) -6 B_3  + 9 B_4) + B_1^5 (7 B_2 ( B_2 + 4)\\
       &\;\; - 24 B_3 + 18 B_4) + 6 B_1^3 ( B_2^3 -2 B_2^2 + 8 B_2 B_3 - 3 B_3^2 + 6 ( B_2 + 1) B_4) -  6 B_1 B_2 (3 B_2^3 \\
       & \;\;- 6 B_3^2 + B_2^2 (4 B_3 -6 ) + 6 B_2 ( B_4 -2 B_3 )) + 18 B_2^2 (-2 B_3^2 + B_2 (( B_2 -2 ) B_2 + 2 B_4))  \\
       & \;\; +  B_1^2 B_2 (  B_2 (B_2 (5 B_2 + 6) - 24 B_3 + 18 B_4) -36 (2 B_3 + B_4) )  \rvert   < 2 \lvert (( B_1 -2) B_1 \\
       & \;\; + 2 B_2) (B_1 ( 2 B_1 + B_2 -3 ) + 3 B_3) (4 B_1^3 + 6 B_2^2 -   B_1^2 ( B_2 + 3) - 3 B_1 B_3) \rvert,   \\
     {\bf{C4}}: & \;\;  0 < (2 B_1 - B_1^2 - 2 B_2)/(2 (B_1 - B_2)) < 1 .
\end{align*}

\end{assumption}
\begin{theorem}\label{A5Ss}
    If $f(z) = z + a_2 z^2 + a_3 z^3 +\cdots \in \mathcal{S}^*_{s}(\varphi)$ and coefficients of $\varphi(z)$ satisfy the conditions {\bf{C1}}, {\bf{C2}}, {\bf{C3}} and {\bf{C4}}, then
    $$ \lvert a_5 \rvert \leq \frac{B_1}{4}.$$
     The bound is sharp.
\end{theorem}
\begin{proof}
    Let $f(z) = z+ \sum_{n=2}^\infty a_n z^n \in \mathcal{S}^*_{s}(\varphi)$, then there exist a Schwarz function $\omega(z)$ such that
    $$ \frac{z f'(z)}{f(z) - f(-z)} = \varphi(\omega(z)).$$
    By the one-to-one correspondence between the class of Schwarz functions and the class $\mathcal{P}$, we obtain
\begin{equation}\label{cftps}
         \frac{z f'(z)}{f(z) - f(-z)} = \varphi\bigg(\frac{p(z)-1}{p(z)+1}\bigg)
\end{equation}
   for some $p(z) =1 + \sum_{n=1}^\infty p_n z^n \in \mathcal{P}$. On the comparison of the same powers of $z$ with the series expansions of functions $f(z)$, $\varphi(z)$ and $p(z)$, the above equation yields
\begin{equation}\label{a5S}
    a_5 = \frac{B_1}{8} (\Upsilon_1 p_1^4 + \Upsilon_2 p_1^2 p_2  + \Upsilon_3 p_1 p_3 +\Upsilon_4 p_1^2 p_2 + p_4),
\end{equation}
   where
\begin{equation}\label{a5Cc}
\left.
\begin{array}{cc}
\begin{aligned}
    \Upsilon_1 &= \frac{ B_1^2 - 2 B_1 + 6 B_2 - 2 B_1 B_2 + B_2^2 - 6 B_3 + 2 B_4}{16 B_1},\\
    \Upsilon_2 &= \frac{3 B_1 -  B_1^2 - 6 B_2 + B_1 B_2 + 3 B_3}{4 B_1}, \;\;\; \Upsilon_3 = \frac{ B_2 - B_1 }{ B_1},\\
    \Upsilon_4 &= \frac{ B_1^2 -2 B_1 + 2 B_2}{4 B_1}.
\end{aligned}
\end{array}
\right\}
\end{equation}
    Let us consider that $q(z) = 1 +\sum_{n=1}^\infty \kappa_n z^n$ and $h(z)=1 + \sum_{n=2}^\infty  \nu_n z^n$ are the members of $\mathcal{P}$, then by Lemma \ref{L2} for $p \in \mathcal{P}$, we have
\begin{equation}\label{RA_4}
     1+H(z):= 1+ \sum_{n=1}^\infty \frac{p_n \kappa_n}{2}z^n  \in \mathcal{P}.
\end{equation}
   For $h \in \mathcal{P} $ and the function $1 + H(z)$ given in (\ref{RA_4}), Lemma \ref{A_4} gives
\begin{equation}\label{A41}
    A_4=\frac{1}{2}\gamma_0 \kappa_4 p_4 -\frac{1}{4}\gamma_1 \kappa_2^2 p_2^2 -\frac{1}{2}\gamma_1 \kappa_1 \kappa_3 p_1 p_3  + \frac{3}{8}\gamma_2 \kappa_1^2 \kappa_2 p_1^2 p_2 -\frac{1}{16}\gamma_3 \kappa_1^4 p_1^4 ,
\end{equation}
   where $\gamma_0=1$,
\begin{equation}\label{svne}
      \gamma_1 = \frac{1}{2} \bigg( 1 + \frac{1}{2} \nu_1 \bigg), \;\;\gamma_2 = \frac{1}{4} \bigg( 1 + \nu_1 + \frac{1}{2} \nu_2 \bigg), \;\; \gamma_3 = \frac{1}{8} \bigg(1 + \frac{3}{2} \nu_1 + \frac{3}{2} \nu_2 + \frac{1}{2} \nu_3 \bigg)
\end{equation}
    and
\begin{equation}\label{bnd}
     \lvert A_4 \rvert \leq 2.
\end{equation}
   Now, in order to establish the required bound, we construct functions $h(z)$ and $q(z)$ such that
\begin{equation}\label{rln}
    A_4 = \Upsilon_1 p_1^4 + \Upsilon_2 p_1^2 p_2  + \Upsilon_3 p_1 p_3 +\Upsilon_4 p_1^2 p_2 + p_4,
\end{equation}
   where $\Upsilon$'s and $A_4$ are given in (\ref{a5Cc}) and (\ref{A41}) respectively. For $0< \tau <1$, define
   $$ q(z) = \frac{1 + 2 \tau z + 2 \tau^2 z^2 + 2 \tau z^3 + z^4}{1 - z^4}, $$
   which yields
\begin{equation}\label{kappa2}
   \kappa_1 = \kappa_3 = 2 \tau, \quad \kappa_2= 2 \tau^2 \quad \text{and} \quad \kappa_4 = 2.
\end{equation}
   From \cite[Theorem 1]{Bano}, we have $q \in \mathcal{P}$.
   To construct function $h(z)$, using Lemma \ref{Cho's lemma}, let
   $$ h(z)= \frac{1 + \omega_1(z)}{1- \omega_1(z)}$$ 
   such that
\begin{equation}\label{Pr_2}
    \omega_1(z)= z \Psi_{-\varepsilon_1}(z \Psi_{-\varepsilon_2}(\varepsilon_3 z)) 
\end{equation}
   where $\varepsilon_1, \varepsilon_2 \in \mathbb{D}$ and $\varepsilon_3 \in \overline{\mathbb{D}}$.
    Thus, we have
   \begin{align*}
   \nu_1 &= 2\varepsilon_1, \quad  \nu_2 = 2 \varepsilon_1^2 + 2(1-\lvert \varepsilon_1\rvert ^2)\varepsilon_2, \\
   \nu_3 &= 2\varepsilon_1^3 +4(1-\lvert \varepsilon_1\rvert ^2)\varepsilon_1 \varepsilon_2 - 2(1-\lvert \varepsilon_1\rvert ^2) \overline{\varepsilon_1} \varepsilon_2^2 + 2(1-\lvert \varepsilon_1\rvert ^2)(1-\lvert \varepsilon_2\rvert ^2) \varepsilon_3.
\end{align*}
     The above set of equations may be satisfied by many $\varepsilon$'s. For our purpose, we impose some restriction on $\varepsilon$'s and take all $\varepsilon$'s as real numbers. Therefore,
\begin{equation}\label{kappa's}
\begin{array}{cc}
   \left.
\begin{aligned}
     \nu_1 &= 2 \varepsilon_1, \quad \nu_2 = 2 \varepsilon_1^2 + 2(1- \varepsilon_1^2)\varepsilon_2,\\
    \nu_3 &= 2 \varepsilon_1^3 +4(1- \varepsilon_1^2)\varepsilon_1 \varepsilon_2 - 2(1 - \varepsilon_1^2) \varepsilon_1 \varepsilon_2^2 + 2(1 - \varepsilon_1^2)(1 - \varepsilon_2^2) \varepsilon_3.
\end{aligned}
   \right\}
\end{array}
\end{equation}
    In addition, if we define
\begin{equation*}\label{zeta's}
\begin{aligned}
      \varepsilon_1 =&\frac{B_1^3 - 2 B_1 B_2 + 2 B_2^2}{2 B_1^2 - B_1^3 - 2 B_1 B_2} , \quad
      \varepsilon_2 = \frac{B_1^3 - B_1^2 B_2 + 3 B_2^2 - 3 B_1 B_3}{3 (-B_1^2 + B_1^3 + B_2^2)},\\
    \varepsilon_3 =& \bigg(B_1^7 - B_1^6 (8 B_2 + 3) - 6 B_1^4 ( B_2 (3 B_2 + 2 B_3 + 2) -6 B_3 + 9 B_4) + B_1^5 (7 B_2 (B_2 \\
   & + 4) - 24 B_3 + 18 B_4) +  6 B_1^3 ( B_2^3 -2 B_2^2 + 8 B_2 B_3 - 3 B_3^2 + 6 ( B_2 + 1) B_4)\\
   &-  6 B_1 B_2 (3 B_2^3 - 6 B_3^2 + B_2^2 (4 B_3 -6) - 6 B_2 (2 B_3 - B_4)) + 18 B_2^2 (-2 B_3^2  \\
   &+ B_2 ((B_2 -2) B_2 + 2 B_4)) +
   B_1^2 B_2 (-36 (2 B_3 + B_4) +
      B_2 (B_2 (6 + 5 B_2) - 24 B_3 \\
      &+ 18 B_4))\bigg)\bigg/\bigg(2 (( B_1 -2) B_1 +
     2 B_2) (B_1 (2 B_1 + B_2 -3) + 3 B_3) (4 B_1^3 + 6 B_2^2 \\
     &-     B_1^2 (3 + B_2) - 3 B_1 B_3)\bigg)
\end{aligned}
\end{equation*}
    and
    $$ \tau =\sqrt{\frac{2 B_1 - B_1^2 - 2 B_2}{2 (B_1 - B_2)}}, $$
   then by the hypthesis \ref{prp1}, we have $\lvert \varepsilon_1 \rvert < 1$, $\lvert \varepsilon_2 \rvert < 1$, $\lvert \varepsilon_3 \rvert < 1$ and $0< \tau < 1$. Putting these defined $\varepsilon$'s in (\ref{kappa's}), we obtain $\nu_i$'s, which in turn  together with (\ref{svne}) yields
\begin{equation}\label{gammas}
\left.
\begin{array}{cc}
\begin{aligned}
     \gamma_1 = & -\frac{(B_1 - B_2)^2}{B_1 ( B_1^2  -2 B_1 + 2 B_2)},  \\
     \gamma_2 = &  -\frac{(B_1 - B_2)^2 (B_1^2 + 6 B_2 - B_1 (3 + B_2) - 3 B_3)}{ 3 B_1 ( B_1^2 -2 B_1 + 2 B_2)^2}, \\
     \gamma_3 = & -\frac{(B_1 - B_2)^2 (B_1^2 + 6 B_2 + B_2^2 - 2 B_1 (1 + B_2) - 6 B_3 + 2 B_4)}{4 B_1 ( B_1^2 -2 B_1 + 2 B_2)^2}.
\end{aligned}
\end{array}
   \right\}
\end{equation}
     On putting the values of $\kappa_i's$ and $\gamma_i's$  from (\ref{kappa2})  and (\ref{gammas}) respectively in (\ref{A41}), we get (\ref{rln}), which together with (\ref{a5S}). Using the bound $\lvert A_4 \rvert \leq 2$ in (\ref{rln}), we get
     $$ \lvert \Upsilon_1 p_1^4 + \Upsilon_2 p_1^2 p_2  + \Upsilon_3 p_1 p_3 +\Upsilon_4 p_1^2 p_2 + p_4\rvert \leq 2, $$
     which together with (\ref{a5S}) gives the desired bound of $\lvert a_5\rvert $.

     Consider the function $\tilde{f}_5(z) = z + \sum_{n=2}^\infty \tilde{a}_n z^n$ in the unit disk satisfying
     $$ \frac{z \tilde{f}_5'(z)}{\tilde{f}_5(z) - \tilde{f}_5(-z)} = \varphi(z^4),$$
     where $\varphi(z)$ is given by (\ref{phi(z)}). Clearly, $\tilde{f}_5 \in \mathcal{S}^*_s(\varphi)$. Equating the coefficients in the above equation, we obtain
      $ \tilde{a}_2 = \tilde{a}_3 = \tilde{a}_4 =0$ and  $\tilde{a}_5 =B_1/4, $
     that demonstrates the sharpness of the bound.
\end{proof}
    For $-1 \leq B < A \leq 1$, Consider the classes
    $  \mathcal{S}^*_s [A,B] := \mathcal{S}^*((1+A z)/(1+B z))$ and $\mathcal{S}^*_{s,SG}: =\mathcal{S}^*(2/(1+e^{-z}))$. These classes are analogues to the corresponding classes of starlike functions introduced and studied in \cite{Goel,Jano}.
     Theorem \ref{A5Ss} directly gives the following result for these classes.
\begin{corollary}\label{crlSs1}
     If $f(z) = z+ \sum_{n=2}^\infty a_n z^n \in \mathcal{S}^*_{s}[A,B]$ such that $A$ and $B$ satisfy the following conditions
\begin{equation}
\left.
\begin{array}{cccc}
\begin{aligned}
    {\bf{C1:}}  & \; \lvert (A - B)^2 (A + B + 2 B^2)\rvert < \lvert ( A - 3 B -2 ) (A - B)^2 \rvert, \\
    {\bf{C2:}}  & \; \lvert (A - B)^3 ( B + 1) \rvert < 3 \lvert (A - B)^2 ( A -1 + ( B -1 ) B) \rvert,\\
    {\bf{C3:}}  & \;  \lvert (A - B)^5 (B + 1) (A^2 ( 7 B + 1) + B ( B (38 + (12 - 17 B) B) \\
         & + 15 )  + A (  B ( B ( 5 B -31 ) -27 ) -3 )) \rvert <  2 \lvert (A - B)^4 ( A\\
          & - 3 B -2)  (A ( B  -2 )  - 4 B^2  + 2 B + 3  ) (A ( B + 4) +  2 B ( B\\
          & -2 )  -3 )\rvert, \\
     {\bf{C4:}}  & \; 0 < (3 B - A + 2 )/(2 B + 2) < 1,
\end{aligned}
\end{array}
\right\}
\end{equation}
   then $$\lvert a_5 \rvert \leq (A-B)/4.$$ The bound is sharp.
\end{corollary}
\begin{example}
   For $A=0$ and $B =-1/2$, all conditions in \ref{crlSs1} are satisfied. Thus, if $f(z) = z+ \sum_{n=2}^\infty a_n z^n \in \mathcal{S}^*_{s}[0,-1/2]$, then $\lvert a_5 \rvert \leq 1/8$.
\end{example}
\begin{corollary}
    If $f (z) = z+ \sum_{n=2}^\infty a_n z^n \in \mathcal{S}^*_{s,SG}$, then $ \lvert a_5 \rvert \leq 1/8$ and the bound is sharp.
\end{corollary}
   In case of the classes $\mathcal{S}^*_{s,L}$ and $\mathcal{S}^*_{s,RL}$, the coefficients of corresponding $\varphi$ satisfy the conditions {\bf{C1}}, {\bf{C2}}, {\bf{C3}}  and {\bf{C4}}. Theorem \ref{A5Ss} yields the following result for these classes:
\begin{remark}
    If $f (z) = z+ \sum_{n=2}^\infty a_n z^n  \in \mathcal{S}^*_{s,L}$, then $\lvert a_5 \rvert \leq 1/8 $ \cite[Theorem 5(a)]{Khatt}.
\end{remark}
\begin{remark}
    If $f (z) = z+ \sum_{n=2}^\infty a_n z^n  \in \mathcal{S}^*_{s,RL}$, then $\lvert a_5 \rvert \leq (5 - 3 \sqrt{2})/8 $ \cite[Theorem 5(b)]{Khatt}.
\end{remark}
\begin{theorem}\label{A5Cc}
    If $f(z)= z + a_2 z^2 + a_3 z^3 +\cdots \in \mathcal{C}_{s}(\varphi)$ and coefficients of $\varphi(z)$ satisfy the conditions {\bf{C1}}, {\bf{C2}}, {\bf{C3}}  and {\bf{C4}}, then
    $$ \lvert a_5 \rvert \leq \frac{B_1}{20}.$$
    The bound is sharp.
\end{theorem}
\begin{proof}
    Let $f(z) = z + \sum_{n=2}^\infty a_n z^n \in \mathcal{C}_{s}(\varphi)$, then there exists a Schwarz function $\omega(z)$ such that
    $$ \frac{(2 z f'(z))'}{(f(z) - f(-z))'} = \varphi(\omega(z)). $$
     Corresponding to the Schwarz function $\omega(z)$, let there is a function $p(z) = 1 + \sum_{n=1}^\infty p_n z^n \in \mathcal{P}$ satisfying $p(z)=(1 + \omega(z))/(1 - \omega(z))$. Thus, we obtain
\begin{equation}\label{anCc}
      \frac{(2 z f'(z))'}{(f(z) - f(-z))'} = \varphi\bigg(\frac{1 - p(z)}{ 1 + p(z)}\bigg).
\end{equation}
  Comparing the coefficients of the same powers of $z$ after applying the series expansion of $f(z)$, $\varphi(z)$ and $p(z)$ leads to
     $$   a_5 = \frac{B_1}{20} (\Upsilon_1 p_1^4 + \Upsilon_2 p_1^2 p_2  + \Upsilon_3 p_1 p_3 +\Upsilon_4 p_1^2 p_2 + p_4),$$
     where $\Upsilon_i$'s are given in (\ref{a5Cc}). Since, $\Upsilon_i$'s are the same as in the case of $\mathcal{S}^*_s(\varphi)$, therefore following the same methodology as in Theorem 1, we get the bound of $\lvert a_5\rvert.$

     To see the sharpness, consider the function $\tilde{g}_5(z) = z + \sum_{n=2}^\infty \tilde{a}_n z^n$ in $\mathbb{D}$ such that
     $$ \frac{(2 z \tilde{g}_5'(z))'}{(\tilde{g}_5(z) - \tilde{g}_5(-z))'} = \varphi(z^4). $$
     Comparison of coefficients of same powers yield $\tilde{a}_2=\tilde{a}_3 =\tilde{a}_4 =0$ and $\tilde{a}_5 = B_1/20$, which proves the sharpness of the bound.
\end{proof}
     We can define the classes $\mathcal{C}_s[A,B]$, $\mathcal{C}_{s,e}$, $\mathcal{C}_{s,SG}$, $\mathcal{C}_{s,L}$ and $\mathcal{C}_{s,RL}$ in a similar manner as 
     $\mathcal{S}^*_s[A,B]$, $\mathcal{S}^*_{s,e}$,  $\mathcal{S}^*_{s,SG}$, $\mathcal{S}^*_{s,L}$ and $\mathcal{S}^*_{s,RL}$ respectively. For these classes, Theorem \ref{A5Cc} yields the following:
\begin{corollary}
\begin{enumerate}[(i)]
  \item If $f(z) = z +\sum_{n=2}^\infty a_n z^n \in \mathcal{C}_s[A,B]$ such that $A$ and $B$ satisfy the conditions given in (\ref{crlSs1}), then $\lvert a_5 \rvert \leq (A-B)/20.$
  \item If $f(z) = z +\sum_{n=2}^\infty a_n z^n \in \mathcal{C}_{s,e}$, then $\lvert a_5 \rvert \leq 1/20$.
  \item If $f(z) = z +\sum_{n=2}^\infty a_n z^n \in \mathcal{C}_{s,L}$, then $\lvert a_5 \rvert \leq 1/40$.
  \item If $f(z) = z +\sum_{n=2}^\infty a_n z^n \in \mathcal{C}_{s,RL}$, then $\lvert a_5 \rvert \leq  (5 - 3 \sqrt{2})/40 $.
  \item If $f(z) = z +\sum_{n=2}^\infty a_n z^n \in \mathcal{C}_{s,SG}$, then $\lvert a_5 \rvert \leq  1/40$.
\end{enumerate}
   All these bounds are sharp.
\end{corollary}
\section{Hermitian-Toeplitz Determinant}
     Shanmugam et al. \cite{a3Ss} obtained the bound of $\lvert a_3 - \mu a_2^2 \rvert$ for $f(z) = z + \sum_{n=2}^\infty a_n z^n$ belonging to the classes $\mathcal{S}^*_s(\varphi)$ and $\mathcal{C}_s(\varphi)$. For $\mu =0$, the following bounds directly follow, which helps us to prove the results:
\begin{lemma}\label{la3Ss}
     If $f(z) =  z+ a_2 z^2 +a_3 z^3+ \cdots \in \mathcal{S}^*_s(\varphi)$ and $B_1 \leq \lvert B_2 \rvert$, then
\begin{align*}
     \lvert a_3 \rvert \leq \frac{\lvert B_2\rvert}{2}.
\end{align*}
\end{lemma}
\begin{lemma}\label{la3Cs}
    If $f(z) =  z+ a_2 z^2 +a_3 z^3+ \cdots \in \mathcal{C}_s(\varphi)$ and $B_1 \leq \lvert B_2 \rvert$, then
\begin{align*}
     \lvert a_3 \rvert \leq \frac{\lvert B_2\rvert}{6}.
\end{align*}
\end{lemma}
\begin{theorem}\label{HtdSs}
   If $f \in \mathcal{S}^*_{s}(\varphi)$ and $B_1 \leq \lvert B_2 \rvert$, then
   $$ T_{3,1}(f) \leq 1.$$
   The bound is sharp.
\end{theorem}
\begin{proof}
   Let $f(z) = z + \sum_{n=2}^\infty a_n z^n \in \mathcal{S}^*_{s}(\varphi)$, then
\begin{equation}
    T_{3,1}(f) = 1 + 2 \lvert a_2\rvert^2 -\lvert a_3 \rvert^2 + 2 \RE(a_2^2 \bar{a}_3).
\end{equation}   Applying the inequality $2 \RE(a_2^2 \bar{a}_3) \leq 2 \lvert a_2^2\rvert \lvert a_3 \rvert$ in the last equation, we obtain
   $$T_{3,1}(f) \leq 1 + 2 \lvert a_2\rvert^2 -\lvert a_3 \rvert^2 +2 \lvert a_2^2\rvert \lvert a_3 \rvert =: g(x), $$
   where $g(x) = 1 + 2 \lvert a_2\rvert^2 - x^2  +2 \lvert a_2^2\rvert x$ with $x =\lvert a_3 \rvert. $
    For $f \in \mathcal{S}^*_s(\varphi)$, we have $\lvert a_2 \rvert \leq B_1/2$ and from Lemma \ref{la3Ss}, $\lvert a_3 \rvert \leq \lvert B_2 \rvert /2.$ Thus $\lvert a_2 \rvert \in [0, 1]$ and $x = \lvert a_3 \rvert \in [0,1]$. As $g'(x) =0 $ at $x =\lvert a_2 \rvert^2$ and $g''(x) < 0$ for all $x \in [0,1]$. Consequently, we have
\begin{align*}
           T_{3,1}(f) &\leq \max{ g(x)} \\
           &=  g(\lvert a_2\rvert^2)
           = (\lvert a_2 \rvert^2 -1)^2
            \leq 1.
\end{align*}
    Since the identity function $f(z) =z$ is a member of the class $\mathcal{S}^*_s(\varphi)$ and for this function, we have $a_2 =0$, $a_3 =0$ and $T_{3,1}(f) =1$, which shows that the bound is sharp.
\end{proof}
\begin{theorem}\label{HtdCc}
   If $f\in \mathcal{C}_{c}(\varphi)$  and $B_1 \leq \lvert B_2 \rvert$, then
   $$ T_{3,1}(f) \leq 1.$$
   The result is sharp.
\end{theorem}
\begin{proof}
    Let $f(z) = z+ \sum_{n=2}^\infty a_n z^n \in \mathcal{C}_s(\varphi)$, then using the inequality $\RE (a_2^2 \bar{a}_3) \leq \lvert a_2 \rvert^2 \lvert a_3 \rvert$ in (\ref{htd}) for $f \in \mathcal{C}_s(\varphi)$, we obtain
   $$ T_{3,1}(f) \leq 1 + 2 \lvert a_2\rvert^2 -\lvert a_3 \rvert^2 + 2 \lvert a_2 \rvert^2 \lvert a_3 \rvert =: g(x), $$
   where $g(x) = 1 + 2 \lvert a_2\rvert^2 - x^2 + 2 \lvert a_2 \rvert^2 x$. Since $ \lvert a_2 \rvert \leq {B_1}/{4}$ and from Lemma \ref{la3Cs}, we have $\lvert a_3 \rvert \leq {\lvert B_2 \rvert}/{6} $, therefore $ \lvert a_2 \rvert \in [0 , 1/2]$ and $\lvert a_3 \rvert \in [0, 1/3]$. Also, note that $g(x)$ attains its maximum value at $x = \lvert a_2 \rvert^2$. Hence
\begin{align*}
           T_{3,1}(f) &\leq \max{ g(x)} \\
           &=  g(\lvert a_2\rvert^2)
           = (\lvert a_2 \rvert^2 -1)^2
            \leq 1.
\end{align*}
    The equality case holds for $f(z)=z.$
\end{proof}
\begin{theorem}\label{LHtdSs}
   If $f \in \mathcal{S}^*_{s}(\varphi)$ such that $B_1^2 > 2 B_2$, then the following estimates hold:
\begin{align*}
   T_{3,1}(f) \geq &
\left\{
\begin{array}{lll}
    \min \bigg\{ 1-\dfrac{B_1^2}{4}, 1 - \dfrac{B_1^2}{2} + \dfrac{B_1^2 B_2}{4} - \dfrac{B_2^2}{4}  \bigg\}, & \sigma_1 \notin [0,4], \\  \\
     1 - \dfrac{B_1^2}{2} + \dfrac{B_1^2 B_2}{4} - \dfrac{B_2^2}{4} , & \sigma_1=4, \\ \\
      1 -\dfrac{B_1^3 (B_1^3 + 4 B_1^2  -4 B_1 - 8 B_2)}{16 (B_1^3 + B_1^2 ( B_2 -1) - 2 B_1 B_2 - B_2^2)} , &\sigma_1 \in (0,4),
\end{array}
\right.
\end{align*}
   where
   $$ \sigma_1 = \frac{2 B_1(B_1^2 - 2 B_2)}{( B_1^2 -B_1 - B_2) (B_1 + B_2)}.$$
    First two inequalities are sharp.
\end{theorem}
\begin{proof}
   Let $f(z) =  z +\sum_{n=2}^\infty a_n z^n \in \mathcal{S}^*_s(\varphi)$, then from (\ref{cftps}), we obtain
\begin{align*}
  a_2 =  \frac{B_1 p_1}{4} \quad \text{and} \quad a_3 = \frac{1}{8} (-B_1 p_1^2 + B_2 p_1^2 + 2 B_1 p_2).
\end{align*}
   Since the class $\mathcal{S}^*_s(\varphi)$ and the class $\mathcal{P}$ is rotationally invariant, therefore we can take $p_1 = p \in [0,2]$. Moreover, Libera et al. \cite{Libera} showed that $2 p_2 = p_1^2 + (4 - p_1^2) \zeta$, $\zeta \in \mathbb{D}$ for $p(z) = 1 + \sum_{n=1}^\infty p_n z^n \in \mathcal{P}$. Thus, we have
\begin{align*}
   - \lvert a_3 \rvert^2 &= -\frac{1}{64} \bigg(B_2^2 p_1^4 + B_1^2 (4-p_1^2)^2 \lvert\zeta\rvert^2 + 2 B_1 B_2 p_1^2 (4-p_1^2) \RE \bar{\zeta} \bigg), \\
   2 \RE (a_2^2 \bar{a}_3) & = \frac{1}{64} B_1^2 p_1^2 \bigg(( B_2 -B_1 ) p_1^2 + B_1 (p_1^2 + (4 - p_1^2) \RE \bar{\zeta}) \bigg).
\end{align*}
   Taking these into account in (\ref{T31}), we get
\begin{align*}
   T_{3,1}(f) &=  \frac{1}{64} \bigg((B_1^2 - B_2) B_2 p_1^4 -  B_1^2 (4 - p_1^2)^2 \lvert\zeta\rvert^2 + B_1 (B_1^2 - 2  B_2) p_1^2 (4 - p_1^2) \RE \bar{\zeta} \bigg)\\
          & \;\;\;\;\; - \frac{B_1^2 p_1^2}{8} + 1 =: F(p_1, \lvert \zeta \rvert, \RE \bar{\zeta}).
\end{align*}
   It can be seen that $F(p_1, \lvert \zeta \rvert, \RE \bar{\zeta}) \geq F(p_1, \lvert \zeta \rvert, - \lvert \zeta \rvert)=:G(x, y)$ by considering $p_1^2 =x$ and $\lvert \zeta \rvert=y $, where
\begin{align*}
    G(x, y) & = \frac{1}{64} \bigg((B_1^2 - B_2) B_2 x^2 -  B_1^2 (4 - x)^2 y^2 - B_1 (B_1^2 - 2  B_2) x (4 - x) y \bigg)\\
          & \;\;\;\;\; - \frac{B_1^2 x}{8} + 1 .
\end{align*}
   Whenever $B_1^2 > 2 B_2$, we have
   $$\frac{\partial G}{\partial y} = \frac{1}{64}( - 2 B_1^2 (4 - x)^2 y - B_1 (B_1^2 - 2 B_2) x (4-x)) \leq 0$$
   for $x \in [0,4]$ and $y \in [0,1]$, which means that $G(x,y)$ is a decreasing function of $y$ and $G(x,y) \geq G(x,1) =:I(x)$ with
   $$ I(x)=  \frac{1}{64} (B_1^3 + B_1^2 ( B_2 -1 ) - 2 B_1 B_2 - B_2^2) x^2 + \frac{B_1}{16} ( 2  B_2 -B_1^2) x - \frac{B_1^2}{4} +1.$$
   An easy computation yields that $I'(x) =0$ at
   $$ x_0 = \frac{2 B_1(B_1^2 - 2 B_2)}{( B_1^2 -B_1 - B_2) (B_1 + B_2)}$$
   and
   $$ I''(x_0) = \frac{1}{32} ( B_1^2 - B_1  - B_2) (B_1 + B_2). $$
   Since $B_1^2 > 2 B_2$, therefore numerator of $x_0$ is always positive. Moreover, denominator of $x_0$ and numerator of $G''(x_0)$ are same, therefore $x_0 <0$ (or $x_0 > 0$) iff $I''(x_0)<0$ (or $I''(x_0) > 0$). Here we discuss the following cases: \\
   {\bf{Case I:}} Whenever $x_0 \in (0,4)$, then $I''(x_0) >0$. Thus $I(x)$ attains its minimum value at $x_0$, which gives
\begin{align*}
         T_{3,1}(f) &\geq I(x_0) \\
                    &=1 -\frac{B_1^3 (B_1^3 + 4 B_1^2  -4 B_1 - 8 B_2)}{
 16 (B_1^3 + B_1^2 ( B_2 -1) - 2 B_1 B_2 - B_2^2)}
\end{align*}
 {\bf{Case II:}} When $x_0 <0$ or $x_0 >4$, which indicates that $I(x)$ does not have any critical point, therefore
\begin{align*}
    T_{3,1}(f) &\geq \min\{I(0), I(4) \}\\
                  & = \min\bigg\{ 1- \frac{B_1^2}{4}, 1 - \frac{B_1^2}{2} + \frac{B_1^2 B_2}{4} - \frac{B_2^2}{4} \bigg\}.
\end{align*}
    For $x_0=4$,
     $T_{3,1}(f) \geq I(4).$

    Function $\tilde{f}_2 \in \mathcal{S}^*_s(\varphi)$ and $\tilde{f}_3 \in \mathcal{S}^*_s(\varphi)$ given by
    $$  \frac{z \tilde{f}_2'(z)}{\tilde{f}_2(z) - \tilde{f}_2(-z)} = \varphi(z), \quad \frac{z \tilde{f}_3'(z)}{\tilde{f}_3(z) - \tilde{f}_3(-z)} = \varphi(z^2)$$
    shows that these bounds are sharp as
    $$ T_{3,1}(\tilde{f}_2) =1 - \frac{B_1^2}{2} + \frac{B_1^2 B_2}{4} - \frac{B_2^2}{4} \quad \text{and} \quad  T_{3,1}(\tilde{f}_3) = 1 - \frac{B_1^2}{4},$$
   which completes the proof.
\end{proof}
\begin{theorem}\label{LHtdCc}
   If $f\in \mathcal{C}_s(\varphi)$ and $3 B_1^2 \geq 8 B_2$, then the following estimates hold:
\begin{align*}
   T_{3,1}(f) \geq &
\left\{
\begin{array}{lll}
    \min \bigg\{ 1-\dfrac{B_1^2}{144}, 1 + \dfrac{1}{144} (3 B_1^2 ( B_2 -6) - 4 B_2^2) \bigg\}, & \sigma_2 \notin [0,4], \\  \\
      1 + \dfrac{1}{144} (3 B_1^2 ( B_2 -6) - 4 B_2^2) , & \sigma_2=4, \\ \\
      1 - \dfrac{B_1^2 ( 9 B_1^4 + 114 B_1^3 + 289 B_1^2  - 304 B_1 B_2 - 36 B_1^2 B_2 + 48 B_2^2)}{576 (3 B_1^3 - 8 B_1 B_2 + 3 B_1^2 B_2 - 4 B_2^2)}  , &\sigma_2 \in (0,4),
\end{array}
\right.
\end{align*}
   where
   $$ \sigma_2 = \frac{2 (17 B_1^2 + 3 B_1^3 - 8 B_1 B_2)}{3 B_1^3 - 8 B_1 B_2 + 3 B_1^2 B_2 - 4 B_2^2}.$$
    First two inequalities are sharp.
\end{theorem}
\begin{proof}
     Let $f(z) = z+ \sum_{n=2}^\infty a_n z^n \in \mathcal{C}_s(\varphi)$, then from (\ref{anCc}), we obtain
\begin{equation}\label{a2a3Cc}
       a_2 = \frac{B_1 p_1}{8}, \quad  a_3 = \frac{1}{24} ( (B_2 -B_1) p_1^2 + 2 B_1 p_2).
\end{equation}     The rotationally invariant property of the classes $\mathcal{C}_s(\varphi)$ and $\mathcal{P}$ allows to take  $p_1\in [0,2]$. Using the formula $2 p_2 = p_1^2 +(4 - p_1^2)\zeta$ (see \cite{Libera}) in (\ref{a2a3Cc}), we get
\begin{align*}
     - \lvert a_3 \rvert^2 &= -\frac{1}{576} (B_2^2 p_1^4 + B_1^2 ( 4 - p_1^2)^2 \lvert\zeta\rvert^2 + 2 B_1 B_2 p_1^2 (4 - p_1^2) \RE\bar{\zeta}),\\
      2 \RE(a_2^2 \bar{a}_3) &=\frac{B_1^2 p_1^2}{768}  (-B_1 p_1^2 + B_2 p_1^2 + B_1 (p_1^2 + (4 - p_1^2) \RE\bar\zeta)).
\end{align*}
     These above values together with (\ref{T31}) leads to
\begin{align*}
    T_{3,1}(f) = & \bigg(\frac{B_1^2 B_2}{768} - \frac{B_2^2}{576} \bigg) p_1^4 - \frac{1}{576} B_1^2 ( 4 - p_1^2)^2 \lvert\zeta\rvert^2 + \bigg(\frac{3 B_1^3 - 8 B_1 B_2}{2304}\bigg) p_1^2 (4 - p_1^2) \RE \bar{\zeta} \\
         & - \frac{B_1^2 p_1^2}{32}+ 1 =: F(p_1, \lvert \zeta\rvert, \RE \bar{\zeta}).
\end{align*}
   As $\RE\bar{\zeta}\geq - \lvert \zeta \rvert$, hence $F(p_1, \lvert \zeta\rvert, \RE \bar{\zeta})\geq F(p_1, \lvert \zeta\rvert, -\lvert\zeta\rvert):= G(x,y)$, where
\begin{align*}
    G(x,y) = &\bigg(\frac{B_1^2 B_2}{768} - \frac{B_2^2}{576} \bigg) x^2 - \frac{1}{576} B_1^2 ( 4 - x)^2 y^2 - \bigg(\frac{3 B_1^3 - 8 B_1 B_2}{2304}\bigg) x(4 - x) y \\
         & - \frac{B_1^2 p_1^2}{32}+ 1
\end{align*}
  for $x = p_1^2 \in [0.4]$ and $y = \lvert\zeta\rvert\in [0,1].$  Whenever $3 B_1^2 \geq 8 B_1 B_2$, we have
  $$ \frac{\partial G(x,y)}{\partial y} = - \frac{1}{288} B_1^2 ( 4 - x)^2 y^2 - \bigg(\frac{3 B_1^3 - 8 B_1 B_2}{2304}\bigg) x(4 - x) y \leq 0.$$
   Therefore, $G(x,y)$ is decreasing function of $y$ and $G(x,y) \geq G(x,1)=:I(x)$, where
\begin{align*}
    I(x)= \frac{(3 B_1^3 - 8 B_1 B_2 +
    3 B_1^2 B_2 - 4 B_2^2)}{2304}x^2  - \frac{B_1 (17 B_1 + 3 B_1^2 - 8 B_2)}{576} x  - \frac{B_1^2}{144} + 1.
\end{align*}
   An elementary calculation reveals that $I'(x)=0$ at
   $$ x_0=\frac{2 (17 B_1^2 + 3 B_1^3 - 8 B_1 B_2)}{3 B_1^3 - 8 B_1 B_2 + 3 B_1^2 B_2 - 4 B_2^2}, $$
   and
   $$ I''(x)= \frac{3 B_1^3 - 8 B_1 B_2 + 3 B_1^2 B_2 - 4 B_2^2}{1152}. $$
   Since $3 B_1^2 \geq 8 B_2$ and $B_1 >0$, therefore numerator of $x_0$ is always positive. Also, note that, denominator $x_0$ and numerator of $I''(x)$ is same, therefore sign of $x_0$ and $I''(x)$ changes simultaneously. Here, two cases arise: \\
{\bf{Case I:}} When $0< x_0 <4$. In this case $I''(x)>0$, so the minimum of $I(x)$ attains at $x_0$, which gives
\begin{align*}
   T_{3,1}(f) &\geq I(x_0)\\
              &=  1 - \frac{B_1^2 ( 9 B_1^4 + 114 B_1^3 + 289 B_1^2  - 304 B_1 B_2 - 36 B_1^2 B_2 + 48 B_2^2)}{576 (3 B_1^3 - 8 B_1 B_2 + 3 B_1^2 B_2 - 4 B_2^2)}.
\end{align*}
{\bf{Case II:}} When $x_0 <0$ or $x_0 > 4$, that means $I(x)$ has no critical point. Thus
\begin{align*}
   T_{3,1}(f) &\geq \min \{I(0),I(4) \}\\
              & = \min \bigg\{1 - \frac{B_1^2}{144}, 1 - \frac{B_1^2}{8} + \frac{B_1^2 B_2}{48} - \frac{B_2^2}{36}   \bigg\}.
\end{align*}
    For the case $x_0 =4$, we have
    $ T_{3,1}(f) \geq I(4)$.

    The sharpness of these bounds follows from the functions  $\tilde{g}_2(z)$ and $\tilde{g}_3(z)$ defined by
    $$ \frac{(2 z \tilde{g}_2'(z))'}{(\tilde{g}_2(z) - \tilde{g}_2(-z))'} = \varphi(z), \quad  \frac{(2 z \tilde{g}_3'(z))'}{(\tilde{g}_3(z) - \tilde{g}_3(-z))'} = \varphi(z^2) .$$
    Since
    $$ T_{3,1}(\tilde{g}_2) =1 - \frac{B_1^2}{8} + \frac{B_1^2 B_2}{48} - \frac{B_2^2}{36}, \quad T_{3,1}(\tilde{g}_2) = 1 - \frac{B_1^2}{144},$$
    which completes the proof.
\end{proof}
\section{Some Special Cases}
     If $\varphi(z) = (1 + A z)/(1 + B z)$, the classes $\mathcal{S}^*_s(\varphi)$ and $\mathcal{C}_s(\varphi)$ reduces to the classes $\mathcal{S}^*_s[A,B]$ and $\mathcal{C}_s[A,B]$ respectively.   Theorem \ref{HtdSs} and $\ref{HtdCc}$ immediately give the following sharp bound for the class $\mathcal{S}^*_s[A,B]$ and $\mathcal{C}_s[A,B]$.
\begin{corollary}
\begin{enumerate}[(i)]
  \item  If $f \in \mathcal{S}^*_s[A,B]$ and $A - B  \leq \lvert B^2 - A B  \rvert $, then $T_{3,1}(f) \leq 1.$
  \item If $f\in \mathcal{C}_{s}[A,B]$ and $A - B  \leq \lvert B^2 - A B  \rvert $, then $T_{3,1}(f) \leq 1.$
\end{enumerate}
\end{corollary}
   Theorem \ref{LHtdSs} and \ref{LHtdCc} yield the following lower bound of $T_{3,1}(f)$ for these classes.
\begin{corollary}
    If $f \in \mathcal{S}^*_{s}[A,B]$ such that $A^2 - B^2 > 0$, then the following estimates hold:
\begin{enumerate}
   \item If $(A - B)^2 (1 - A)  (1 - B) < 0$ or $2 (A - B)^2 ( A ( 2 B -1) - B +2) > 0 $, then
\begin{equation*}
     \det{T_{3,1}(f)} \geq \min\left\{ 1-\frac{(A-B)^2}{4}, 1 -  \frac{(A - B)^2  A B + 2 )}{4} \right\}.
\end{equation*}
     \item If $  2 (A - B)^2 ( A ( 2 B -1) - B + 2) = 0$, then
\begin{equation*}
   \det{T_{3,1}(f)}\geq 1 -  \frac{(A - B)^2  A B + 2 )}{4}.
\end{equation*}
   \item If $0 < 2 (A - B)^2 (A + B) <  4 (A - B)^2  (1 - A)  (1 - B) $, then
\begin{equation*}
     \det{T_{3,1}(f)}\geq  1 + \frac{(A - B)^2 ( A^2 + B^2 + 4 B - 2 A ( B -2 )   -4 )}{16 (1 - A) (1 - B)}
\end{equation*}
\end{enumerate}
    First two inequalities are sharp.
\end{corollary}
\begin{corollary}
   If $f\in \mathcal{C}_c(\varphi)$ and $3 A^2 + 2 A B - 5 B^2 \geq 0$, then the following estimates hold:
\begin{enumerate}
   \item If $-(A - B)^2 (3 A ( B -1 ) + B ( B -5 ) ) < 0$ or $2 (A - B)^2 ( 2 B^2  - 5 B + A ( 6 B -3 ) + 17   ) > 0 $, then
\begin{equation}\label{ffff}
     \det{T_{3,1}(f)} \geq \min\left\{ 1-\frac{(A - B)^2}{144}, 1 - \frac{(A - B)^2 ( B^2 + 3 A B +18  )}{144}\right\}.
\end{equation}
     \item If $ (A - B)^2 ( 2 B^2  - 5 B + A ( 6 B -3 ) + 17   )  = 0$, then
\begin{equation}\label{lowerb4}
   \det{T_{3,1}(f)}\geq  1 - \frac{(A - B)^2 ( B^2 + 3 A B +18  )}{144}.
\end{equation}\label{lowerb5}
   \item If $0 < 2 (A - B) (3 A^2 + A (2 B + 17) - B ( 5 B + 17)) < 4 (A - B)^2 (3 A (1 - B) + (5 - B) B)$, then
\begin{equation}\label{lowerb3}
     \det{T_{3,1}(f)}\geq 1 + \frac{(A - B)^2 (9 A^2 + 21 B^2 + 190 B  +
   6 A (3 B + 19 ) + 289 )}{576 (3 A ( B -1) + ( B -5) B)} .
\end{equation}
\end{enumerate}
    First two inequalities are sharp.
\end{corollary}

     For $\varphi(z) = (1+ (1-2\alpha) z )/(1-z)$ and $(1+z)/(1-z)$ in $\mathcal{S}^*_{s}(\varphi)$, we obtain the class $\mathcal{S}^*_{s}(\alpha)$ and Sakaguchi's class, $\mathcal{S}^*_{s}$  respectively, where  $\alpha\in[0,1]$. For more detail of these classes, we refer \cite{mocanu,thanga}. Theorem \ref{HtdSs} and \ref{LHtdSs} yield the following sharp lower and upper bound of $T_{3,1}(f)$ for these classes, proved by Kumar and Kumar~\cite{KK}.
\begin{remark}
\begin{enumerate}[(i)]
  \item  If $f\in \mathcal{S}^*_{s}(\alpha)$, then $ (3 - 2\alpha) \alpha^2 \leq  T_{3,1}(f) \leq 1 $ \cite[Theorem 2.2]{KK}.
  \item  If $f\in \mathcal{S}^*_{s}$, then $  0 \leq T_{3,1}(f) \leq 1 $ \cite[Corollary 2.3]{KK}.
\end{enumerate}
\end{remark}
   For other subclasses of $\mathcal{S}^*_s$, the following sharp bounds follow from Theorem~\ref{LHtdSs}.
\begin{corollary}
   If $f\in \mathcal{S}^*_{s,SG}$, then $T_{3,1}(f) \geq 2009/2304 .$
\end{corollary}
\begin{remark}
\begin{enumerate}[(i)]
    \item If $f\in \mathcal{S}^*_{s,L}$, then $T_{3,1}(f) \geq 221/256$ \cite[Theorem 3.1]{KK}.
  \item If $f\in \mathcal{S}^*_{s,RL}$, then $T_{3,1}(f) \geq (863 - 444 \sqrt{2})/256$ \cite[Theorem 3.3]{KK}.
\end{enumerate}
\end{remark}
   Theorem \ref{HtdCc} and \ref{LHtdCc} give the following corollaries for different subclasses of $\mathcal{C}_c$.
\begin{corollary}
\begin{enumerate}[(i)]
  \item If $f\in \mathcal{C}_{s}[A,B]$ and $A - B  \leq \lvert B^2 - A B  \rvert $, then $T_{3,1}(f) \leq 1.$
  \item If $f\in \mathcal{C}_s(\alpha)$, then $    T_{3,1}(f) \leq 1.$
  \item If $f\in \mathcal{C}_s$, then $  T_{3,1}(f) \leq 1.$ \\
  All these bounds are sharp.
\end{enumerate}
\end{corollary}
\begin{corollary}
\begin{enumerate}[(i)]
  \item  If $f\in \mathcal{C}_{s,SG}$, then $T_{3,1}(f) \geq 40165/41472.$
    \item If $f\in \mathcal{C}_{s,L}$, then $T_{3,1}(f) \geq 4459/4608 $.
  \item If $f\in \mathcal{C}_{s,RL}$, then $T_{3,1}(f) \geq (-3731 + 5835 \sqrt{2})/4608 $.
\end{enumerate}
   All these bounds are sharp.
\end{corollary}
\section*{Declarations}
\subsection*{Funding}
The work of the Surya Giri is supported by University Grant Commission, New-Delhi, India  under UGC-Ref. No. 1112/(CSIR-UGC NET JUNE 2019).
\subsection*{Conflict of interest}
	The authors declare that they have no conflict of interest.
\subsection*{Author Contribution}
    Each author contributed equally to the research and preparation of manuscript.
\subsection*{Data Availability} Not Applicable.
\noindent

\end{document}